\documentclass{birkau}

\usepackage{amsmath,amsthm,amssymb,tikz,stmaryrd}

\theoremstyle{plain}
\newtheorem{theorem}{Theorem}[section]
\newtheorem{lemma}[theorem]{Lemma}

\theoremstyle{definition}
\newtheorem{definition}[theorem]{Definition}

\newtheorem{remark}[theorem]{Remark}

\newcommand{\co}[1]{\mathord{\textsf{#1}}}
\newcommand{\Cl}{\co{G}}
\newcommand{\cl}{\co{cl}}
\newcommand{\id}{\co{id}}

\newcommand{\QQ}{\mathord{\mathbb{Q}}}

\newcommand{\II}{\mathord{\mathbb{I}}}

\newcommand{\fromto}[3]{{#1}\colon{#2}\to{#3}}
\newcommand{\pfromto}[3]{{#1}\colon{#2}\looparrowright{#3}}

\newcommand{\da}{{\mathord{\downarrow}}}
\newcommand{\ua}{{\mathord{\uparrow}}}

\newcommand{\reln}[1]{\lceil#1\rceil}
\newcommand{\relnup}[1]{\lfloor#1\rfloor}

\newcommand{\Meet}{\bigwedge}

\renewcommand{\Join}{\bigvee}

\newcommand{\cat}[1]{\mathord{\mathbf{#1}}}
\newcommand{\SUP}{\cat{SUP}}
\newcommand{\INF}{\cat{INF}}
\newcommand{\CLat}{\cat{CLat}}
\newcommand{\Pol}{\cat{Pol}}

\newcommand{\st}{\,\mid\,}
\newcommand{\defeq}{\mathrel{:=}}
\newcommand{\defiff}{\mathrel{\,:\Longleftrightarrow\,}}
\newcommand{\pol}{\co{pol}}
\newcommand{\rel}{\co{rel}}
\newcommand{\powerset}{\mathord{\mathcal{P}}}

\newcommand{\ignore}[1]{}
\title{A Relational Category of Birkhoff Polarities}
\author{M.~Andrew~Moshier\\
Chapman University}
\date{\today}
\begin{document}

\begin{abstract}
    Garret Birkhoff observed that any binary relation between two sets determines a Galois connection between the powersets, or equivalently, closure operators on the powersets, or equivalently, complete lattices of subsets that are dually isomorphic. 
    Referring to the duality of, say, points and lines in projective geometry, he named the binary relations \emph{polarities}. 

    Researchers since then have used polarities 
        (also known as formal contexts) as a convenient technical 
        way to build complete lattices from ``found'' data. 
    And so, various proposals for suitable morphisms between 
    polarities have tended to have a particular application in mind.

    In this work, we develop the structure of a category 
        of polarities and compatible relations, 
        adopting Birkhoff's original simple idea that the structure 
        of a polarity is its the Galois connection. 
    Hence, morphisms must be relations that, 
        in a reasonable sense, preserve Galois connections. 
    In particular, the dual equivalence of the category to the category of complete meet semilattices, completeness of the category, characterization of epimorphisms and monomorphisms, an epi/mono factorization system, as well as the star-autonomous structure of the category, all arise by extending Birkhoff's original observation to morphisms.
\end{abstract}

\maketitle

\section{Introduction}

Following Garrett Birkhoff \cite{Birkhoff_lattice_1948}, 
    a \emph{polarity} is a binary relation between two sets, 
    determining a complete lattice as explained below.
A typical and motivating example is the polarity obtained 
    from a finite lattice by taking the join irreducibles 
    as one set, the meet irreducibles as the second set,
    and the less than or equals relation between them.
These data are enough to reconstruct an isomorphic copy 
    of the lattice.
Or starting from any partially ordered set, 
    take two copies of set and the less than or equal relation.
These data determine the MacNeille completion 
    (the injective hull) of the poset. 

Polarities are known more recently thanks to the concept 
    analysis community \cite{ganter_formal_1998} 
    as \emph{formal contexts}, 
    where the standard notation is $(G,M;I)$ or $(O,A;I)$ 
    standing for Gegenst\"anden (Objects), 
    Merkmalen (Attributes) and Inzidenz (Incidence), 
    as formal contexts are meant by the concept analysis community 
    as a formalization of how objects can be related 
    to their attributes.
Though \emph{(formal) context} is the common contemporary name, 
    we use the name \emph{polarity} here to remind the reader 
    that the idea dates earlier to Birkhoff.
Birkhoff chose the term `polarity' from an analogy with the duality 
    between points (poles) and polar lines in projective geometry. 
Indeed, the self-dual nature of polarities will play a role 
    in this paper that is obscured by thinking of polarities 
    as comprised of objects and attributes.
To emphasize the dual nature of polarities, 
    we name a context by its incidence relation 
    and mark the underlying sets
    as in $\mathcal{A} = (A^-, A^+; \mathcal A)$.

Polarities determine complete lattices, and all complete lattices
    arise up to isomorphism from polarities.
But complete lattices are, of course, the same structures 
    as complete meet-semilattices and as complete join-semilattices.
The differences only show up when morphisms are considered. 
In this paper, we investigate a notion of morphism 
    for polarities that is, so to speak, ``native'' 
    to Birkhoff's idea of taking polarities as objects of study.
We  answer the question of what constitutes a general morphism 
    by taking the incidence relation of a polarity 
    as its identity morphism, 
    and extending that to general morphisms. 
The category obtained this way is quite natural 
    and admits interesting purely combinatorical methods 
    of construction for its categorical structure.

In the formal context literature one finds various discussions
    of possible ways to formulate a category of polarities.
The main author who has looked at possible morphisms 
    of polarities is Marcel Ern\'e 
    \cite{erne_categories_2014,erne_distributive_1993}. 
This work deals with morphisms from polarity $\mathcal A$ 
    to polarity $\mathcal B$ as pairs
    of maps $\fromto {f^-}{A^-}{B^-}$ and $\fromto {f^+}{A^+}{B^+}$
    satisfying obvious compatibility conditions.  
Ern\'e defines a \emph{conceptual pair} to be such a pair of maps 
    $(f^-,f^+)$ that jointly preserves the incidence relation 
    and for which the pre-image under $f^-$ of a Galois closed 
    subset of $B^-$ is Galois closed in $A^-$ and likewise, 
    the pre-image under $f^+$ of a Galois closed subset $B^+$ 
    is closed in $A^+$. 
The main point is that the obvious construction of a polarity 
    from a complete lattice, sending $L$ to $(L,L,\leq)$,  
    is functorial and has a left adjoint that exhibits 
    complete lattices as a coreflective subcategory 
    of Ern\'e's category of polarities. 
Polarities and conceptual pairs  via the coreflection are 
    turned into complete lattices 
    and complete lattice homomorphisms. 
The paper \cite{erne_categories_2014} provides a wealth 
    of other details and an elegant classification 
    of various special morphisms of complete lattices in terms 
    of their manifestation in maps between polarities.
Much of that paper is motivated by his earlier work 
    \cite{erne_distributive_1993}, in which Ern\'e considered 
    characterizations of polarities, the corresponding 
    complete lattices of which satisfy various distributive laws.

An alternative to Ern\'e is to regard a polarity 
    as a Chu space (over $2$). 
Chu space morphisms are pairs of functions 
    $\fromto {f^-}{A^-}{B^-}$ and $\fromto {f^+}{B^+}{A^+}$
    also satisfying certain compatibility conditions.
Chu spaces, of course, have a significant literature of their own. 
See Vineet Gupta's PhD thesis \cite{gupta_chu_1994} 
    for an early application,
    and Vaughan Pratt's excellent survey \cite{pratt_chu_1999} 
    for an introduction and bibliography.

More recently, Robert Goldblatt \cite{goldblatt_morphisms_2020}
    develops a category of polarities in the spirit of Ern\'e 
    with pairs of functions that satisfy a natural ``back and forth''
    condition generalizing bounded morphisms that are familiar 
    in the literature on algebraic semantics of modal logic. 

G.Q. Zhang, et al, \cite{hitzler_cartesian_2004,zhang_chu_2003},
    also considers categories in which the objects are 
    essentially polarities. 
But in these papers, the morphisms are chosen to yield 
    a category equivalent to the category of information systems,
    hence dually equivalent to Scott domains. 
The morphisms do not have directly to do 
    with the usual interpretation of a polarity, nor a Chu space,
    as specifying a complete lattice. 
In effect, these papers use the terms ``Chu space'' 
    and ``context'' simply to mean ``binary relation'' 
    and then investigate a category that has little to do 
    with either Chu spaces or polarities. 

In \cite{gehrke_generalized_2006}, Gehrke investigates 
    a special class of polarities the author calls 
    \emph{RS-frames}.
These are meant as generalizations of Kripke frames 
    in the sense that one has ``worlds'' (as in a Kripke frame) 
    and ``co-worlds''. 
The S in RS stands for \emph{separated}. 
This is a fairly harmless condition that is roughly
    analogous to being a $T_0$ space 
    -- worlds and co-worlds are ``separated'' by each other.
The R stands for \emph{reduced}. 
This is a substantive condition which translates, again roughly,
    to say that the worlds and co-worlds are, respectively, 
    join and meet irreducible.
In classical modal logic, a world (that is an element of a Kripke 
    model) spells out the truth and falsity of each 
    proposition, and thus determines a prime filter 
    of propositions.
RS-frames capture a similar idea in more generality.
Polarities, regarded as generalized Kripke frames,
    feature prominently algebraic proofs of cut elimination 
    for substructural logic as in \cite{galatos_residuated_2013}.

RS-frames constitute a very special case of polarities. 
The complete lattices that RS-frames determine are \emph{perfect},
    i.e., they are join generated by their completely join 
    irreducibles and meet generated by their
    completely meet irreducibles. 
Of course, such lattices are important for Gehrke's application
    to modal logic, but perfection is a rare property of complete lattices. 
Indeed, the unit interval of reals regarded as a complete lattice fails 
    to have any completely join irreducibles or any completely
    meet irreducibles. 
So an RS-frame cannot describe this lattice.

We propose to consider general polarities with morphisms based 
    on the idea that the incidence relation is itself 
    the identity morphism on the object. 
The Galois connection determined 
    by a polarity is the relevant structure. 
Thus a morphism must also be a binary relation 
    between the lower set of the domain polarity 
    and the upper set of the codomain polarity. 
Once we understand what composition must do, 
    the morphisms are defined as those relations for which 
    the incidence relations on the domain and codomain polarities 
    act as identities.

The initial work reported here was first developed during a visit 
    to St.~Anne's College, Oxford as a Plumer Fellow in 2011 and 2012.
The author thanks Hilary Priestley for being such 
    a gracious host and valued colleague.

\section{The category of polarities}

In this section, we present the category of polarities 
    and relations, and prove a few useful technical results. 

For order theory, we follow the naming and general notation in 
    \textit{Continuous Lattics and Domains} 
    \cite{gierz_continuous_2003}. 
In particular, $\SUP$ is the category of complete lattices 
    and supremum preserving functions, $\INF$, the category of 
    complete lattices and infimum preserving functions, 
    and $\CLat$ is the category of complete lattices 
    and functions that preserve both infima and suprema. 

Functions into an ordered structure are always regarded 
    as being pointwise ordered, unless explicitly described otherwise.
In particular, antitone maps between ordered structures 
    $P$ and $Q$ are ordered as if they are monotone maps 
    from the order opposite of $P$ to $Q$. 

The left adjoint of a monotone function $\fromto fPQ$ if it exists,
    is denoted by $\fromto {f_*}PQ$. 
Likewise, the right adjoint of a monotone function $\fromto gQP$,
    if it exists, is denoted by $\fromto {g^*}LM$. 
The reader will recall that morphisms in $\INF$ always 
    have left adjoints in $\SUP$, and vice versa.

For poset $P$, the order opposite of poset $P$ is denoted by $P^\partial$. 
Thus the equivalence of the categories $\INF$ 
    and $\SUP$ sends $L$ to $L^\partial$
    and leaves $\fromto fLM$ alone as a concrete function.
The dual equivalence from $\INF$ to $\SUP$ leaves objects 
    (complete lattices) alone, 
    and sends $\fromto fLM$ to $\fromto {f_*}ML$.

The well-known correspondence between closure operators and 
    closure systems plays a part in the work, 
    so a quick reminder is useful. 
A \emph{closure operator} on a set $A$ is a function 
    $\fromto c{\powerset(A)}{\powerset(A)}$ that is monotonic,
    inflationary (meaning $X\subseteq c(X)$ holds for all $X$), 
    and idempotent.
A \emph{closure system} is a family $M\subseteq \powerset(A)$ 
    that is closed under arbitrary intersections.
The set of fixpoints of a closure operator 
    ($\co{M}_c = \{X\in \powerset(A) \st c(X) = X\}$ ) 
    is a closed system.
For a closure system $M\subseteq \powerset(A)$, 
    the function $\textsf{c}_M$ defined by $X\mapsto \bigcap\{Y\in M\st X\subseteq Y\}$ 
    is a closure operator. 
The move from $c$ to $\co{M}_c$ and from $M$ to $\co{c}_M$ 
    are inverses. 

For a binary relation $R\subseteq A\times B$, 
    we write $R^\intercal$ for its converse relation; 
    $R[a]$ for the usual ``forward image'' consisting 
    of all $\beta\in B$ related to $a$; and 
    $\fromto{R^\da}{\powerset(B)}{\powerset(A)}$ and
    $\fromto{R^\ua}{\powerset(A)}{\powerset(B)}$ 
    for the antitone maps defined by
\begin{align*}
    R^\da(Y) &=\bigcap_{\beta\in Y} R^\intercal[\beta]\\
    R^\ua(X) &=\bigcap_{a\in X} R[a].
\end{align*}
So, it is easy to see that the following are all equivalent:
\begin{itemize}
    \item $X\times Y\subseteq R$;
    \item $Y\subseteq R^\ua(X)$;
    \item $X\subseteq R^\da(Y)$;
    \item $Y\subseteq (R^\intercal)^\da(X)$; and
    \item $X\subseteq (R^\intercal)^\ua(Y)$.
\end{itemize}

\begin{remark}
    As an aid to type checking, we generally use lower case 
        Latin letters for elements in the domain of a relation 
        and lower case Greek letters for elements of the codomain.
\end{remark}

The functions $R^\da$ and $R^\ua$, thus, 
    form a Galois connection on the subsets of $A$ and $B$.
So both $R^\da$ and $R^\ua$ send arbitrary unions to intersections,
    and both composites $R^\da R^\ua$ and $R^\ua R^\da$ 
    are closure operators on the domain and codomain of $R$, respectively.

Since $\powerset(B)^\partial$ is the free $\INF$ object 
    over the set $B$, any antitone map 
    $\fromto f{\powerset(B)}{\powerset(A)}$ (regarded as a function 
    from $\powerset(B)^\partial$) to $\powerset(A)$ that happens 
    to send unions --- meets in $\powerset(B)^\partial)$  ---
    to intersections arises uniquely as $R^\da$ 
    for some binary relation between $A$ and $B$.

Thus the set of relations $\powerset(A\times B)$ 
    and the hom set $\INF(\powerset(B)^\partial,\powerset(A))$ 
    are in a bijection. 
One direction sends $R$ to $R^\da$.
For the other direction, 
    suppose $\fromto f{\powerset(B)}{\powerset(A)}$ is 
    an antitone map, 
    not necessarily sending unions to intersections.
Let $\reln{f}$ denote the binary relation 
    $\reln{f}\subseteq A\times B$ defined by 
    $a\mathrel{\reln{f}}\beta$ if and only if $a\in f(\{\beta\})$,
    and $\relnup{f}$ for the converse of $\reln{f}$.

\begin{lemma}\label{lem:bracket-da}
    The operation $R\mapsto R^\da$ is right adjoint 
        to the operation $f\mapsto \reln{f}$, 
        in the sense that $\reln{f}\subseteq R$ if and only if 
        $f\leq R^\da$.
    Moreover, $\reln{R^\da} = R$.
    Likewise, $R\mapsto R^\ua$ is right adjoint to 
        $f\mapsto \relnup{f}$, and $\relnup{R^\ua} = R$.

    An antitone map $\fromto f{\powerset(B)}{\powerset(A)}$ 
        sends arbitrary unions to intersections 
        if and only if $f=\reln{f}^\da$.
\end{lemma}
\begin{proof}
    The adjunction facts are trivial to check.
    As noted, $R^\da$ sends unions to intersections 
        for any binary relation $R$.
    If $f$ sends unions to intersections, 
        its behavior is determined by its behavior on singletons.
    But $f$ and $\reln{f}^\da$ agree on singletons.
\end{proof}

\begin{lemma}
    For binary relations $R,S\subseteq A\times B$, 
        the following are  equivalent:
    \begin{itemize}
    \item $R\subseteq S$;
    \item $R^\da \leq S^\da$;
    \item $R^\ua \leq S^\ua$.
    \end{itemize}
\end{lemma}
\begin{proof}
    $R^\da \leq S^\da$ if and only if $\reln{R^\da} \subseteq S$,
        and $R= \reln{R^\da}$.
    Likewise, $R^\ua \leq S^\ua$ if and only if 
        $\relnup{R^\ua} \subseteq S$, and $R=\relnup{R^\ua}$.
\end{proof}

Any $R\subseteq A\times B$ can be regarded as forming 
    a polarity $(X,Y,R)$.
We occasionally need to distinguish between $R$ as a ``raw'' 
    binary relation and $(X,Y,R)$ as a polarity, 
    so for emphasis we may write $\pol(R)$ for the polarity.
And for a polarity $\mathcal A$, we may write $\rel(\mathcal A)$ to highlight 
	the incidence relation as a relation.

For a polarity $\mathcal A$, let $\cl_{\mathcal A}$ denote 
    the composite operation ${\mathcal A}^\da {\mathcal A}^\ua$,
    and $\cl^{\mathcal A}$ denote 
    ${\mathcal A}^\ua {\mathcal A}^\da$.
The complete lattices of fixpoints are $\Cl^-(\mathcal A)$ 
    and $\Cl^+(\mathcal A)$. 
Note that $\Cl^-(\mathcal A)$ is isomorphic to 
    $\Cl^+(\mathcal A)^\partial$ with ${\mathcal A}^\da$ 
    and ${\mathcal A}^\ua$ cutting down to the isomorphisms.

For mere type checking reasons, these considerations suggest 
    that a general morphism from $\mathcal{A}$ to $\mathcal{B}$
    should be a binary relation $R\subseteq A^-\times B^+$ 
    that respects the structure of $\mathcal A$ and $\mathcal B$.
The essential equality that gives rise to the complete lattice
    $\Cl^-(\mathcal A)$ is 
    ${\mathcal A}^\da {\mathcal A}^\ua {\mathcal A}^\da 
        = {\mathcal A}^\da$, 
    written twice to emphasize associativity, is 
\begin{align*} 
    \cl_{\mathcal A}{\mathcal A}^\da &= {\mathcal A}^\da\\
    &= {\mathcal A}^\da \cl^{\mathcal A}
\end{align*}
So, we require two analogous compatibility conditions 
    on a relation $R\subseteq A^-\times B^+$:
\begin{align*}
    \cl_{\mathcal A} R^\da&= R^\da\\
    &= R^\da \cl^{\mathcal B}
\end{align*}
Say that $R$ is \emph{compatible with $\mathcal{A}$ on the left} 
    if the first equality holds, 
    is \emph{compatible with $\mathcal{B}$ on the right} 
    if the second holds, and is simply 
    \emph{compatible with $\mathcal A$ and $\mathcal B$} 
    if both. 
In this case, we write $\pfromto{R}{\mathcal A}{\mathcal B}$ 
    to indicate compatibility.
The following lemma provides other tests for compatibility.

\begin{lemma}\label{lem:compatibility}
    For polarity $\mathcal A$, set $Y$ and relation
        $R\subseteq A^-\times Y$, the following are equivalent:
    \begin{enumerate}
    \item $\cl_{\mathcal A} \leq R^\da R^\ua$;
    \item $\cl_{\mathcal A} R^\da \leq R^\da$;
    \item $R^\ua \leq R^\ua \cl_{\mathcal A}$; 
    \item $X\times Y\subseteq R$ implies
        $\cl_{\mathcal A}(X)\times Y\subseteq R$
    \item $\reln{\cl_{\mathcal A} R^\da} \subseteq R$; and 
    \item $R$ is compatible with $\mathcal A$ on the left.
    \end{enumerate}
    Likewise, for polarity $\mathcal B$, 
        set $X$ and relation $R\subseteq X\times B^+$, 
        the following are equivalent.
    \begin{enumerate}
    \item $\cl^{\mathcal B} \leq R^\ua R^\da$;
    \item $\cl^{\mathcal B} R^\ua \leq R^\ua$;
    \item $R^\da \leq R^\da \cl^{\mathcal B}$; 
    \item $X\times Y\subseteq R$ implies 
        $X\times \cl^{\mathcal B}(Y)\subseteq R$
    \item $\relnup{\cl^{\mathcal B} R^\ua} \subseteq R$; and
    \item $R$ is compatible with $\mathcal B$ on the right
    \end{enumerate}  
\end{lemma}
\begin{proof}
    The second set of equivalences is obtained from the first 
        by replacing $R$ with $R^\intercal$, 
        and ${\mathcal A}$ with ${\mathcal B}^\intercal$.
    So we only need to prove the first set.
    
    (1) implies (2) because ${\mathcal A}^\da {\mathcal A}^\ua$ 
        is monotonic and $R^\da = R^\da R^\ua R^\da$. 
    If (2), then ${\mathcal A}^\da {\mathcal A}^\ua R^\da R^\ua
        \leq R^\da R^\ua$. 
    But ${\mathcal A}^\da {\mathcal A}^\ua$ is monotone 
        and $R^\da R^\ua$ is inflationary, 
        so ${\mathcal A}^\da {\mathcal A}^\ua 
            \leq {\mathcal A}^\da {\mathcal A}^\ua R^\da R^\ua$.
    Since $R$ determines a Galois connection (1) and (3)
        are equivalent.

    Equivalence of (2) and (4) is due to the fact that 
        $X\times Y\subseteq R$ is equivalent to 
        $X\subseteq R^\da(Y)$.
    Equivalence of (2) and (5) is due to the adjunction 
        between $\reln{-}$ and $(-)^\da$.
    Finally, since ${\mathcal A}^\da{\mathcal A}^\ua$ 
        is inflationary, (2) and (6) are equivalent.
\end{proof}

Condition (4) of the lemma along with its corresponding condition 
    for compatibility on right means that 
    $R\subseteq A^-\times B^+$ is compatible if and only if 
    $X\times Y\subseteq R$ implies 
    $\cl_{\mathcal A}(X)\times \cl^{\mathcal B}(Y) \subseteq R$. 

One can picture a compatible relation between polarities 
    in a simple graphical way.
We indicate a binary relation as a non-horizontal line drawn
    between two (names of) sets:

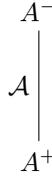
\[
\begin{tikzpicture}
 \node(Xp) at (0,1) {$A^-$};
 \node(Xm) at (0,-1) {$A^+$};
 \draw[-] (Xm.north) to node[left] {${\mathcal A}$} (Xp.south);
\end{tikzpicture}
\]

We let the ``altitude'' of the two sets indicate the domain 
    and codomain of the relation.
The relative horizontal position does not matter, 
    but usually a relation drawn as a vertical line indicates 
    a polarity (an object in our category).
So the above is a picture of the polarity $\mathcal{A}$.

Lemma \ref{lem:compatibility}(1) tells us that compatibility 
    of $R$ with $\mathcal{A}$ on the left means that 
    ${\mathcal A}^\da {\mathcal A}^\ua \leq R^\da R^\ua$. 
So we decorate vertices in the lower part of a diagram 
    with $\leq$ or $\geq$ to indicate this relation.
Using similar decorations we can indicate compatibilities 
    on the right.
With these conventions, a relation $R$ that is compatible 
    with $\mathcal{A}$ on the left and $\mathcal{B}$ 
    on the right can be pictured as

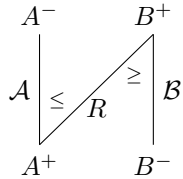
\[
\begin{tikzpicture}
 \node(Xp) at (0,1) {$A^-$};
 \node(Xm) at (0,-1) {$A^+$};
 \node(Yp) at (1.5,1) {$B^+$};
 \node(Ym) at (1.5,-1) {$B^-$};

 \draw[-] (Xm.north) to node[right, pos=.4] {${}_\leq$} node[left] {$\mathcal A$} (Xp.south);
 \draw[-] (Xm.north) to node[below] {$R$} (Yp.south);
 \draw[-] (Ym.north) to node[left, pos=.6] {${}^\geq$} node[right] {$\mathcal B$} (Yp.south);
\end{tikzpicture}
\]

We now have the data to define the category $\Pol$ having 
    as objects all polarities and as morphisms from 
    $\mathcal{A}$ to $\mathcal{B}$ all binary relations 
    $R\subseteq A^-\times B^+$
    satisfying the compatibility conditions. 
Of course, it remains to define composition in the category, 
    and check that it works. 

\begin{lemma}\label{lem:cxt-composition}
    Suppose $\pfromto R{\mathcal A}{\mathcal B}$ and 
        $\pfromto{S}{\mathcal B}{\mathcal C}$ 
        are compatible relations as indicated. 
    Then $R^\da \mathcal B^\ua S^\da$ sends unions in 
        $\powerset(C^+)$ to intersections in $\powerset(A^-)$.
    Moreover, the relation $\reln{R^\da \mathcal B^\ua S^\da}$ 
        is compatible with $\mathcal{A}$ on the left 
        and $\mathcal{C}$ on the right.
\end{lemma}
\begin{proof}
    Let $\{Z_i\}_i$ be a family of subsets of $C^+$.
    Then the identities $R^\da \mathcal B^\da\mathcal B^\ua = R^\da$ and $\mathcal B^\da\mathcal B^\ua S^\da$, 
    	together with the fact that $\mathcal R^\da$, $\mathcal B^\da$, and $S^\da$  
    	send unions to intersections yield the calculation
    \begin{align*}
        R^\da {\mathcal B}^\ua S^\da(\bigcup_i Z_i) 
        &= R^\da {\mathcal B}^\ua(\bigcap_i S^\da(Z_i))\\
        &= R^\da {\mathcal B}^\ua(\bigcap_i {\mathcal B}^\da {\mathcal B}^\ua S^\da (Z_i))\\
        & = R^\da {\mathcal B}^\ua {\mathcal B}^\da (\bigcup_i {\mathcal B}^\ua S^\da(Z_i))\\
        &= R^\da(\bigcup_i {\mathcal B}^\ua S^\da(Z_i))\\
        &= \bigcap_i(R^\da {\mathcal B}^\ua S^\da(Z_i))
    \end{align*}

    Compatibility of the result with $\mathcal A$ 
        on the left follows from compatibilty of $R$ with 
        $\mathcal A$ on the left, since
    \begin{align*}
    {\mathcal A}^\da {\mathcal A}^\ua \reln{R^\da {\mathcal B}^\ua S^\da}^\da &= {\mathcal A}^\da {\mathcal A}^\ua R^\da {\mathcal B}^\ua S^\da\\
    & = R^\da {\mathcal B}^\ua S^\da\\
    &= \reln{R^\da {\mathcal B}^\ua S^\da}^\da
    \end{align*}
    Compatibility with $\mathcal C$ is proved similarly.
\end{proof}

For relations $R$ and $S$ with compatibility as above, 
    define $R\fatsemi S\defeq\reln{R^\da {\mathcal B}^\ua S^\da}$.  
According to the foregoing lemma,
    $(R\fatsemi S)^\da = R^\da {\mathcal B}^\ua S^\da$.  
Note that we have written composition 
    in the ``diagramatic'' order for relation composition, 
    not the ``applicative'' order for function composition.  
Since composition is defined with respect 
    to a particular identity morphism 
    (the incidence relation for $\mathcal B$), 
    we sometimes write $R\fatsemi_{\mathcal B} S$ 
    to emphasize this.
Composition can be pictured using our diagrams 
    as in Figure \ref{fig:composition}.
\begin{figure}[ht]
  \centering
  \begin{tikzpicture}
    \node(Xp) at (0,1) {$A^-$};
    \node(Xm) at (0,-1) {$A^+$};
    \node(Yp) at (1.5,1) {$B^+$};
    \node(Ym) at (1.5,-1) {$B^-$};
    \node(Zp) at (3,1) {$C^+$};
    \node(Zm) at (3,-1) {$C^-$};

    \draw[-] (Xm.north) to node[right, pos=.4] {${}_{\leq}$} node[left] {${\mathcal A}$} (Xp.south);
    \draw[-] (Xm.north) to node[above] {$R$} (Yp.south);
    \draw[-] (Ym.north) to node[left, pos=.6] {${}^\geq$} node[right, pos=.4] {${}_\leq$} node[below left] {${\mathcal B}$} (Yp.south);
    \draw[-] (Ym.north) to node[above] {$S$} (Zp.south);
    \draw[-] (Zm.north) to node[left,pos=.6] {${}^\geq$} node[right] {${\mathcal C}$} (Zp.south);	
    \end{tikzpicture}
  
  \begin{tikzpicture}
    \node(Xp) at (0,1) {$A^-$};
    \node(Xm) at (0,-1) {$A^+$};
    \node(Zp) at (3,1) {$C^+$};
    \node(Zm) at (3,-1) {$C^-$};

      \draw[-] (Xm.north) to node[right,pos=.35] {${}_\leq$} node[left] {${\mathcal A}$} (Xp.south);
      \draw[-] (Xm.north) to node[above] {$R\fatsemi S$}(Zp.south);
      \draw[-] (Zm.north) to node[right] {${\mathcal C}$} node[left, pos=.65] {${}^\geq$} (Zp.south);
  \end{tikzpicture}
  \caption{Composition of compatible relations}
  \label{fig:composition}
\end{figure}
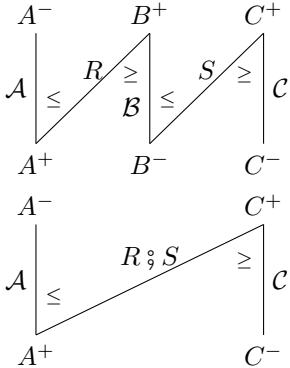

Compatibility with $\mathcal{A}$ on the left and $\mathcal B$ 
    on the right is precisely the property that ensures 
    that a polarity is the identity morphism for 
    $\fatsemi$ composition.
Associativity of composition is easily verified: 
    for compatible relations 
    $\pfromto R{\mathcal A}{\mathcal B}$, 
    $\pfromto S{\mathcal B}{\mathcal C}$ 
    and $\pfromto T{\mathcal C}{\mathcal D}$,
\[
(R\fatsemi S)\fatsemi T = \reln{(R\fatsemi S)^\da {\mathcal C}^\ua T^\da} =
\reln{R^\da {\mathcal B}^\da S^\da {\mathcal C}^\ua T^\da} = R\fatsemi(S\fatsemi T).
\] 
This leads to the definition of our relational category 
    of polarities.

\begin{definition}\label{def:category-cxt}
  The category $\Pol$ consists of polarities as the objects and
  compatible relations as the morphisms. 
  Composition is defined by
  $\fatsemi$.
  The incidence relation $\mathcal A$ is the identity morphism on the polarity $\mathcal A$.
\end{definition}

Clearly, an arbitrary relation $R\subseteq A^-\times B^+$ will not
necessarily be compatible with polarities $\mathcal A$ and
$\mathcal B$. 
One can hope that a suitable compatible relation can
be generated from $R$.

\begin{lemma}\label{lem:cxt-homsets-are-complete}
  For polarities $\mathcal{A}$ and $\mathcal{B}$ the collection of compatible relations 
  	from $\mathcal{A}$ to $\mathcal{B}$ is closed under arbitrary intersection.
\end{lemma}
\begin{proof}
This is obvious from the characterization of compatibility by the condition 
	that $X\times Y\subseteq R$ implies $\cl_{\mathcal A}(X)\times \cl^{\mathcal B}(Y)\subseteq R$.
\end{proof}

So any relation $R\subseteq
A^-\times B^+$ has a smallest compatible relation $R^*$ containing it.
We call this the \emph{compatibilization of $R$}.

\begin{remark}
Compatibilization of a given relation may involve transfinite recursion. 
This is because the operator $c$ sending $R$ to 
\[
c(R) =\bigcup_{X\times Y\subseteq R}\cl_{\mathcal A}(X)\times \cl^{\mathcal B}(Y)
\]
is monotonic and inflationary, but may not be idempotent.
So the process by which $R^*$ is obtained may need to pass through transfinitely 
	many iterates by letting $R^0 = R$, $R^{\rho+1} = c(R^\rho)$ for each ordinal, 
	and $R^\lambda = \bigcup_{\rho<\lambda} R^\rho$ for limits.
Then $R^*$ is $R^\rho$ at the first stage for which this is a fixpoint of $c$.
\end{remark}

The category $\Pol$ is, in fact, an $\INF$ enriched category.

\begin{lemma}
    For any family $\{\pfromto {R_i}{\mathcal A}{\mathcal B}\}_i$ 
    of compatible relations, and any $Y\subseteq B^+$, 
    it is the case that $(\bigcap_i R_i)^\da(Y) = \bigcap_i R_i^\da(Y)$. 
\end{lemma}
\begin{proof}
    For $Y\subseteq  C^+$, $a\in (\bigcap_i R_i)^\da(Y)$
    	if and only if for every $\beta\in Y$ and every $i$, 
    	it is the case that $a \mathrel{R_i}\beta$. 
    So the claim holds by exchanging the quantifiers. 
\end{proof}

\begin{lemma}\label{lem:hom-complete}
	Composition in $\Pol$ preserves arbitrary intersections of compatible relations.
\end{lemma}
\begin{proof}
	Suppose we have compatible relations $\pfromto Q{\mathcal A}{\mathcal B}$, 
		$\pfromto {R_i}{\mathcal B}{\mathcal C}$ for $i\in I$ and $\pfromto S{\mathcal C}{\mathcal D}$. 
    For $Z\subseteq D^+$, 
\begin{align*}
(Q\fatsemi (\bigcap_i R_i) \fatsemi S)^\da(Z) 
&= Q^\da {\mathcal B}^\ua (\bigcap_i R_i)^\da {\mathcal C}^\ua S^\da (Z)\\
&= Q^\da {\mathcal B}^\ua (\bigcap_i R_i^\da {\mathcal C}^\ua S^\da(Z))\\
&= Q^\da {\mathcal B}^\ua (\bigcap_i \mathcal B^\da\mathcal B^\ua R_i^\da {\mathcal C}^\ua S^\da(Z))\\
&= Q^\da {\mathcal B}^\ua \mathcal B^\da (\bigcup_i \mathcal B^\ua R_i^\da {\mathcal C}^\ua S^\da(Z))\\
&= Q^\da (\bigcup_i \mathcal B^\ua R_i^\da {\mathcal C}^\ua S^\da(Z))\\
&= \bigcap_i (Q^\da \mathcal B^\ua R_i^\da \mathcal C^\ua S^\da(Z))\\
&=\bigcap_i (Q\fatsemi R_i\fatsemi S)^\da(Z)\\
&= (\bigcap_i Q_\fatsemi R_i\fatsemi S)^\da (Z).
\end{align*}
\end{proof}

The combinatorial nature of compatibility means that much of 
the categorical structure of $\Pol$ is quite easily described. 

\begin{lemma}\label{lem:cxt-selfdual}
  $\Pol$ is a self dual category.
\end{lemma}
\begin{proof}
For a polarity $\mathcal A$, define ${\mathcal A}^\partial = (A^+, A^-; {\mathcal A}^\intercal)$. 
This is also a polarity.  If $\pfromto R{\mathcal A}{\mathcal B}$ is
a compatible relation, then $R^\intercal$ is obviously
compatible with ${\mathcal B}^\partial$ on the left and ${\mathcal A}^\partial$ on the right. 
So defining
$R^\partial = R^\intercal$ yields a contravariant endofunctor.
Clearly, $R^{\partial\partial} = R$, so this is a dual isomorphism.
\end{proof}

Monomorphisms and epimorphisms in $\Pol$ are characterized by the following simple
conditions.

\begin{lemma}\label{lem:cxt-monos}
  A compatible relation $\pfromto R{\mathcal B}{\mathcal C}$ is a monomorphism if and only if
  $R^\da R^\ua \leq \cl_{\mathcal B}$.
  Similarly, $R$ is an epimorphism if and only if $\cl^{\,\mathcal C}\geq R^\ua R^\da$.
\end{lemma}
\begin{proof}
  Suppose $R^\da R^\ua  \leq {\mathcal B}^\da {\mathcal B}^\ua$.
  Consider compatible relations
  $\pfromto {P,Q}{\mathcal A}{\mathcal B}$ so that $P\fatsemi
  R \subseteq Q\fatsemi R$. Then
  \begin{align*}
    P^\da &= P^\da {\mathcal B}^\ua {\mathcal B}^\da\\
    &= P^\da {\mathcal B}^\ua R^\da R^\ua {\mathcal B}^\da\\
    &\leq Q^\da {\mathcal B}^\ua R^\da R^\ua {\mathcal B}^\da\\
    &\leq Q^\da {\mathcal B}^\ua {\mathcal B}^\da\\
    &= Q^\da
  \end{align*}
So $P\subseteq Q$.

Conversely, suppose $\pfromto R{\mathcal B}{\mathcal C}$ 
    is a compatible relation
    so that 
    \[
    R^\da R^\ua(Y) \neq {\mathcal B}^\da {\mathcal B}^\ua(Y)
    \]
    for some $Y\subseteq B^-$. 
Let $Y'$ be $R^\da R^\ua(Y)$.
Without loss of generality, 
    we can choose $Y$ to be a Galois closed subset of $B^-$. 
So $Y\subseteq Y'$ is a strict inclusion, 
    and by compatibility $Y'$ is also Galois closed.

Let $\II$ denote the polarity 
    $(\{\bullet\}, \{\bullet\}, \emptyset)$.
For $\Upsilon\subseteq B^+$,
    define $Q_{\Upsilon} \subseteq \II^-\times B^+$ by 
    $\bullet \mathrel{P_B} \beta$
    if and only if $\beta\in \Upsilon$.
Then trivially, $Q_\Upsilon$ is compatible with $\II$ on the left.
If $\Upsilon$ is Galois closed, then $\Xi\subseteq \Upsilon$
    implies $\cl^{\mathcal B}(\Xi)\subseteq \Upsilon 
        = Q_\Upsilon^\ua Q_\Upsilon^\da(\Upsilon)$. 
And $\Xi\nsubseteq \Upsilon$ implies 
    $Q_\Upsilon^\ua Q_\Upsilon^\da(\Xi) = B^+$. 
So $Q_\Upsilon$ is also compatible with $\mathcal B$ on the right.

Now consider $Q_{{\mathcal B}^\ua(Y)}$ 
    and $Q_{{\mathcal B}^\ua(Y')}$. 
These are unequal compatible relations from $\II$ to $\mathcal B$.
But using compatibility of $R$ and the fact that 
    $Q_\Upsilon^\ua(\{\bullet\}) = \Upsilon$, 
    the following inclusions are equivalent 
    for any $Z\subseteq C^+$.
\begin{align*}
    \{\bullet\}&\subseteq Q_{\mathcal B^\ua(Y)}^\da 
            {\mathcal B}^\ua R^\da(Z)\\
    {\mathcal B}^\ua R^\da(Z) 
        &\subseteq Q_{{\mathcal B}^\ua(Y)}^\ua(\{\bullet\})\\
    {\mathcal B}^\ua R^\da(Z) &\subseteq {\mathcal B}^\ua(Y)\\
    Y &\subseteq R^\da(Z)\\
    Y'&\subseteq R^\da(Z).
\end{align*}
    So $Q_{{\mathcal B}^\ua(Y)}\fatsemi R 
        = Q_{{\mathcal B}^\ua(Y')}\fatsemi R$.
\end{proof}

This leads to the following epi-mono factorization system 
    for $\Pol$.
Suppose $\pfromto R{\mathcal A}{\mathcal B}$ 
    is a compatible relation. 
Then $\pol(R) = (A^-,B^+, R)$ constitutes another polarity.  
The same relation $R\subseteq A^-\times \pol(R)^+$ is 
    compatible as a relation from $\mathcal A$ to $\pol(R)$, 
    and also compatible as a relation
    from $\pol(R)$ to $\mathcal B$. 
Interpreted in the first way, 
    $R^\ua R^\da = \pol(R)^\ua\pol(R)^\da$ 
    (because $R$ and $\pol(R)$ are the same relation), 
    and in the second $R^\da R^\ua = \pol(R)^\da \pol(R)^\ua$.
So $\pfromto R{\mathcal A}{\pol(R)}$ is an epimorphism, 
    and $\pfromto R{\pol(R)}{\mathcal B}$ is a monomorphism.  
And since $R^\da = R^\da \pol(R)^\ua R^\da$, 
    $R$ factors into $R\fatsemi_{\pol{R}} R$. 
Diagramatically, 
    this is depicted as in Figure \ref{fig:factorizedR}.

\begin{figure}[ht]\label{fig:factorizedR}
	\begin{tikzpicture}
 \node(Xm) at (0,-1) {$A^-$};
 \node(Xp) at (0,1) {$A^+$};
 \node(Yp) at (6,1) {$B^+$};
 \node(Ym) at (6,-1) {$B^-$};

 \draw[-] (Xm.north) to node[right, pos=.4] {${}_\leq$} node[left] {$\mathcal A$} (Xp.south);
 \draw[-] (Xm.north) to node[below] {$R$} (Yp.south);
 \draw[-] (Ym.north) to node[left, pos=.6] {${}^\geq$} node[right] {$\mathcal B$} (Yp.south);
\end{tikzpicture}

\begin{tikzpicture}
 \node(Xm) at (0,-1) {$A^-$};
 \node(Xp) at (0,1) {$A^+$};
 \node(Yp) at (3,1) {$B^+$};
 \node(Ym) at (3,-1) {$A^-$};
 \node(Zp) at (6,1) {$B^+$};
 \node(Zm) at (6,-1) {$B^-$};

 \draw[-] (Xm.north) to node[right, pos=.4] {${}_\leq$} node[left] {$\mathcal A$} (Xp.south);
 \draw[-] (Xm.north) to node[below] {$R$} (Yp.south);
 \draw[-] (Ym.north) to node[left, pos=.7] {${}^=$} node[right,pos=.3] {${}_=$} node[above right] {$\pol(R)$} (Yp.south);
 \draw[-] (Ym.north) to node[below] {$R$} (Zp.south);
 \draw[-] (Zm.north) to node[left, pos=.6] {${}^\geq$} node[right] {$\mathcal B$} (Zp.south);
\end{tikzpicture}
\caption{Epi-mono factoring of a compatible relation}
\end{figure}
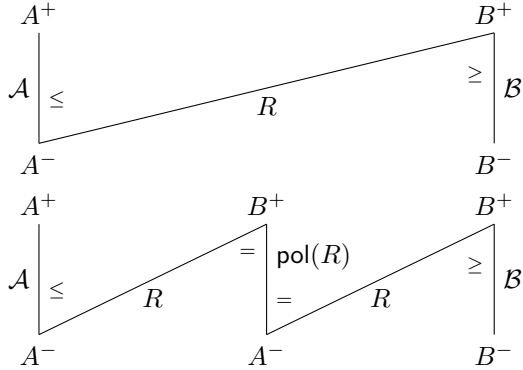

Essentially the same technique shows that a morphism that is both 
    a monomorphism and an epimorphism is necessarily 
    an isomorphism. 

\begin{lemma}\label{lem:Cxt is balanced}
The category $\Pol$ is balanced.
\end{lemma}
\begin{proof}
    Suppose compatible relation 
        $\pfromto R{\mathcal A}{\mathcal B}$ 
        is both a monomorphism and an epimorphism.
By Lemma~\ref{lem:cxt-monos}, 
    then $R^\da R^\ua = {\mathcal A}^\da {\mathcal A}^\ua$ 
    and $R^\ua R^\da = {\mathcal B}^\ua {\mathcal B}^\da$.

These equalities mean that the relation $\rel(\mathcal A)$ is 
    compatible with the polarity $\pol(R)$ on the right, 
    and $\rel(\mathcal B)$ is compatible with $\pol(R)$ 
    on the left. 
Of course, $R$ is compatible with $\pol(R)$. 
So $\pfromto {\rel(\mathcal B)}{\mathcal B}{\pol(R)}$ and 
    $\pfromto {\rel(\mathcal A)}{\pol(R)}{\mathcal A}$ are 
    morphisms that compose through $\pol(R)$ yielding 
    a candidate inverse 
    $R^-=\rel(\mathcal B)\fatsemi_{\pol(R)} \rel(\mathcal A)$. 

The composition $R\fatsemi_{\mathcal B} R^-$ is 
$\reln{R^\da {\mathcal B}^\ua {\mathcal B}^\da R^\ua {\mathcal A}^\da}$. 
Because $R$ is an epimorphism, this equals $\reln{R^\da R^\ua {\mathcal A}^\da}$, 
	and because it is a monomorphism, this equals $\reln{{\mathcal A}^\da}$. 
The composition $R^-\fatsemi_{\mathcal A} R$ equals $\mathcal B$ by a similar calculation.
\end{proof}

It is worth emphasizing that isomorphisms in the category $\Pol$ are
plentiful. 
For example, for a set $A$, define polarities $P(A) = (A,A, \neq)$, and $P'(A) = (A,\powerset(A),\in)$,
and $P''(A) = (\powerset(A),\powerset(A),\between)$, where $\between$ denotes the non-empty
intersection relation between subsets of $A$.  
One can show that
$\mathord{\in}$ restricted to $A$ and $\powerset(A)$ is an isomorphism
from $P(A)$ to $P'(A)$, and also from $P'(A)$ to $P''(A)$.
To illustrate the point farther, consider the following analogue of
$T_0$ separation.

Say that polarity $\mathcal A$ is \emph{$-$-separating} 
	if for each $a,a''\in A^-$, $\mathcal A[a]=\mathcal A[a']$ implies $a=a'$.  
Likewise, say it is $+$-separating if for each $\alpha,\alpha'\in A^+$, 
	 ${\mathcal A}^{\intercal}[\alpha] = {\mathcal A}^{\intercal}[\alpha']$ implies $\alpha=\alpha'$. 
A polarity that is both $-$-separating and $+$-separating is simply called \emph{separating}.

\begin{lemma}\label{lem:separating}
  For any polarity $\mathcal A$, there is a separating polarity isomorphic to $\mathcal A$.
\end{lemma}
\begin{proof}
  Define $\equiv$ as the equivalence relation on $A^-$ given by $a\equiv a'$ if and only if $\mathcal A[a]=\mathcal A[a']$.
  Then define $\hat{\mathcal A} = (A^-/\equiv, A^+; \hat{\mathcal A})$ where $[a]\hat{\mathcal A} \alpha$ if and only if $a \mathrel{\mathcal A}\alpha$.  
  Now it is easy to check that ${\mathcal A}^\ua {\mathcal A}^\da = \hat{{\mathcal A}}^\ua \hat{{\mathcal A}}^\da$.
  So $\hat{{\mathcal A}}$ serves also as a morphism from $\hat{\mathcal A}$ to $\mathcal A$, and ${\mathcal A}$ serves as a morphism from $\mathcal A$ to $\hat{\mathcal A}$.
  The above identities translate to ${\mathcal A} = {\mathcal A}\fatsemi_{\hat{\mathcal A}}\hat{{\mathcal A}}$ and $\hat{{\mathcal A}} = \hat{{\mathcal A}}\fatsemi_{\mathcal A} {\mathcal A}$.
  By construction $\hat{\mathcal A}$ is $-$-separating.
  The dual construction involving $A^+$ instead of $A^-$ yields a $+$-separating object $\check{\mathcal A}$ that isomorphic to $A$.
  Clearly, if $\mathcal A$ is $-$-separating, then $\check{\mathcal A}$ is as well.
\end{proof}

\begin{remark}
    The construction of a separating polarity from $\mathcal A$ is actually a functor, and the full category of separating polarities is equivalent to $\Pol$. We omit the details because we do not need them here.
\end{remark}
The result tells us that every object is isomorphic to its
separating co-reflection.  
Contrast this with topological spaces,
where $T_0$ co-reflection is not generally an isomorphism.

To close this section, we define another equivalent full subcategory of $\Pol$ that will play a part in the analysis of the symmetric monoidal structure of $\Pol$.

Say that a polarity $\mathcal A$ is \emph{standard} if 
$A^+ = \powerset(A^-)$, and the incidence relation is reflexive and transitive in the sense that
\begin{itemize}
    \item if $a \in X$ then $a \mathrel{\mathcal A} X$; and 
    \item if $a\mathrel{\mathcal A} X$ and for each $a'\in X$, $a' \mathrel{\mathcal A} X'$, then $a \mathrel{\mathcal A} X'$.
\end{itemize}

\begin{theorem}\label{lem:standardization}
    $\Pol$ is equivalent to the full subcategory consisting of standard polarities.
\end{theorem}
\begin{proof}
  For polarity $\mathcal A$, 
  let $\tilde{\mathcal A}$ be the polarity $(A^-,\powerset(A^-), \tilde{\mathcal A})$ where $a\mathrel{\tilde{\mathcal A}} X$ if and only if $a\in \cl_{\mathcal A}(X)$.
  Because $\cl_{\mathcal A}$ is inflationary and idempotent,  $\tilde{\mathcal A}$ is a standard polarity.
  By design, for $X\subseteq A^-$, 
  \[
  \cl_{\mathcal A}(X) = \cl_{\tilde{\mathcal A}}(X) = \tilde{\mathcal A}^\intercal[X].
  \]
  So,  $\mathbf{X}\subseteq \powerset(A^-)$, 
  \[\tilde{\mathcal A}^\da(\mathbf X) = \bigcap_{X\in \mathbf X}\cl_{\mathcal A}(X) = {\mathcal B}^\da(\bigcup_{X\in\mathbf X}{\mathcal B}^\ua(X)).
  \]
  Moreover, using the Galois connection of ${\mathcal B}^\da$ and ${\mathcal B}^\ua$, $X'\in \tilde{\mathcal A}^\ua(X)$ if and only if $X\subseteq \cl_{\mathcal A}(X')$ if and only if ${\mathcal B}^\ua(X')\subseteq {\mathcal B}^\ua(X).$

  For a compatible relation $\pfromto R{\mathcal A}{\mathcal B}$,
  define $\tilde{R}\subseteq A^-\times \powerset{B^-}$ 
  by $a\mathrel{\tilde{R}} Y$ if and only if $a\in R^\da\mathcal B^\ua(Y)$. 
  So, $Y\in \tilde{R}^\ua(X)$ if and only if $X\subseteq R^\da{\mathcal B}^\ua(Y)$ if and only if ${\mathcal B}^\da R^\ua(X)\subseteq Y$.
  Hence, $\tilde{R}^\ua(X) \subseteq \tilde{R}^\ua(X')$ if and only if ${\mathcal B}^\da R^\ua(X')\subseteq {\mathcal B}^\da R^\ua(X)$ if and only if $R^\ua(X) \subseteq {\mathcal B}^\ua{\mathcal B}^\da R^\ua(X') = R^\ua(X')$.
  So, $\tilde{R}^\da\tilde{R}^\ua = R^\da R^\ua$.
  From this, it follows that $\tilde{R}$ is compatible with $\tilde{A}$ on the left.

  Furthermore, $a\in \tilde{R}^\da {\tilde{\mathcal B}}^\ua(Y)$
  if and only if ${\tilde{\mathcal B}}^\ua(Y) \subseteq \tilde{R}^\ua(\{a\})$ if and only if for every $Y'$, if $Y\subseteq \cl_{\mathcal B}(Y')$, then $a\in R^\da{\mathcal B}^\ua(Y')$. This is equivalent to requiring that for all $Y'$, if ${\mathcal B}^\ua(Y')\subseteq {\mathcal B}^\ua(Y)$, then $a\in R^\da{\mathcal B}^\ua(Y')$. But $R^\da$ is antitone, 
  so this is equivalent to $a\in R^\da{\mathcal B}^\ua(Y)$. 
  That is, $\tilde{R}^\da \tilde{\mathcal B}^\ua = R^\da {\mathcal B}^\ua$. 
  Consequently, $\tilde{R}^\da \cl^{\mathcal B} = R^\da {\mathcal B}^\ua {\tilde{\mathcal B}}^\da$. And for $\mathbf Y\subseteq \powerset(B^-)$, 
  \[
  R^\da {\mathcal B}^\ua {\tilde{\mathcal B}}^\da(\mathbf Y) = R^\da {\mathcal B}^\ua {\mathcal B}^\da(\bigcup_{Y\in\mathbf Y}{\mathcal B}^\ua(Y)) = \bigcap_{Y\in\mathbf Y} R^\da {\mathcal B}^\ua(X) = \tilde{R}^\da(\mathbf Y).
  \]
  So, $\tilde{R}$ is also compatible with $\tilde{\mathcal B}$ on the right.

  For compositions, consider $\pfromto S{\mathcal B}{\mathcal C}$. Then $a \mathrel{\widetilde{R\fatsemi S}}Z$ if and only if $a\in R^\da{\mathcal B}^\ua S^\da {\mathcal C}^\ua(Z)$.
  And $a \mathrel{\tilde{R}\fatsemi \tilde{S}} Z$ if and only if $a\in \tilde{R}^\da {\tilde{\mathcal B}}^\ua \tilde{S}^\da(\{Z\})$.
  But $\tilde{R}^\da {\tilde{\mathcal B}}^\ua = R^\da {\mathcal B}^\ua$ and $\tilde{S}^\da(\{Z\}) = S^\da {\mathcal C}^\ua(Z)$.
  So $\tilde{-}$ is a functor into the category of standard polarities.

 For $\mathcal A$, the relation $\rel(\mathcal A)$ is evidently compatible with the polarity $\mathcal A$  on the right, and as already noted $\cl_{\mathcal A} = \cl_{\tilde{\mathcal A}}$.
 So $\rel{\mathcal A}$ is a natural isomorphism from $\tilde{\mathcal A}$ to $\mathcal A$. The inverse is $\rel(\tilde{\mathcal A})$ taken as a compatible relation from $\mathcal A$ to $\tilde{\mathcal A}$.
\end{proof}

Recalling the correspondence between closure operators and closure systems, note that standard polarities are yet another manifestation of the same idea. 
A closure operator $c$ on $A$ determines a standard polarity by the relation $a \mathrel{\mathsf A_c} X$ if and only if $a\in c(X)$.
Clearly, any (standard) polarity $\mathcal A$ determines a closure operator $\cl_{\mathcal A}$ on $A^-$.
The move between standard polarities and closure operators is also a bijection.

\section{$\INF$ and polarities}\label{sec:INF}

Here we show that the categories $\Pol$ and $\INF$ are dually equivalent.
For a compatible relation $\pfromto R{\mathcal A}{\mathcal B}$,
    define $\Cl^-(R) \defeq R^\da {\mathcal B}^\ua$. 
Evidently, $\Cl^-(R)$ is monotone as it is a composition of two antitone maps,
 and $\rel(\mathcal A)^\da {\mathcal A}^\ua$ is the identity on the complete lattice $\Cl^-(\mathcal A)$.  
 For compatible relations $\pfromto R{\mathcal A}{\mathcal B}$ 
    and $\pfromto S{\mathcal B}{\mathcal C}$, composition is preserved: 
$\Cl^-(R\fatsemi S) = R^\da {\mathcal B}^\ua S^\da {\mathcal C}^\ua$.  
So $\Cl^-$ is indeed functorial
into the category of monotone maps between $\INF$ objects. It remains
to check that $\Cl^-(R)$ preserves infima.

\begin{lemma}\label{lem:cxt-to-inf}
  The functor $\Cl^-$ sends a compatible relation $R$ to an intersection
  preserving map between Galois closed sets. 
  Hence, it co-restricts to a functor $\Pol^{\textsf{op}}\Rightarrow \INF$.
\end{lemma}
\begin{proof}
Consider polarities $\mathcal A$ and $\mathcal B$ and a family $\{Y_i\}_i$ in $\cl^-(\mathcal B)$,
\begin{align*}
  R^\da {\mathcal B}^\ua (\bigcap_i Y_i)& = R^\da {\mathcal B}^\ua (\bigcap_i {\mathcal B}^\da {\mathcal B}^\ua(Y_i))\\ 
                            & = R^\da {\mathcal B}^\ua {\mathcal B}^\da(\bigcup_i {\mathcal B}^\ua(Y_i))\\  
                            & = R^\da(\bigcup_i {\mathcal B}^\ua(Y_i))\\
                            &= \bigcap_i R^\da {\mathcal B}^\ua(Y_i)
\end{align*}
Compatibility with $\mathcal A$ means that the image of $R^\da {\mathcal A}^\ua$ 
is in $\Cl^-(\mathcal A)$.
\end{proof}

Note that $\Cl^-$ is contravariant in spite of what the notation
suggests, because we have chosen to write composition in $\Pol$
diagramatically (as is appropriate for relations) and in $\INF$ applicatively (as is appropriate for functions).

For any object $L$ in $\INF$, define its \emph{obvious
  polarity} to be 
\[
\co{C}(L)\defeq (L,L,\leq_L).
\]
Clearly, $\mathord{\leq_L^\ua}(X)$
is the collection of all upper bounds of $X$, and similarly
$\mathord{\leq_L^\da}(X)$ is the collection of all lower bounds. 
So, $\cl_{\co{C}(L)}(X) = \da \Join X$
and $\cl^{\co{C}(L)}(X) = \ua\Meet X$.

For an $\INF$ morphism $\fromto hML$, define the relation
$\co{C}(h)\subseteq L\times M$ by 
\[
a\mathrel{\co{C}(h)} b \defiff a\leq_L
h(b).
\] 
Hence $\co{C}(\id_L) = \leq_L$.
Moreover, $\co{C}(h)^\da(Y) = \da h(\Meet Y)$, and
$\co{C}(h)^\ua(X) = \{b\in M\st \forall a\in X. a\leq_L
h(b)\}$.  
So, $\co{C}(h)^\ua(X) = \ua h_*(\Join X)$ -- recalling that $h_*$ is the left adjoint of $h$.
In summary, $\co{C}(h)$ is compatible with the polarities
$\co{C}(L)$ and $\co{C}(M)$.

\begin{lemma}\label{lem:cxt-to-inf-op}
  $\co{C}$ constitutes a contravariant functor from $\INF$ to $\Pol$.
\end{lemma}
\begin{proof}
  All that remains is to check that composition is
  preserved. 
  Suppose $\fromto gNM$ and $\fromto hML$ are morphisms in $\INF$.
  Because $hg$ preserves infima, $a\in \co{C}(hg)^\da(Z)$
  if and only if $a\leq hg(\Meet Z)$ for any $Z\subseteq N$.
  And because $g$ preserves infima, $a\in \co{C}(h)^\da \leq_M^\ua \co{C}(g)^\da(Z)$
  if and only if $a\leq h(b)$ for every $b\in M$
  satisfying $g(\Meet Z)\leq b$. 
\end{proof}

Now we have the ingredients for our main theorem.

\begin{theorem}\label{thm:cxt=inf-op}
  The category $\Pol$ is dually equivalent to the category $\INF$.
\end{theorem}
\begin{proof}
  For a given polarity $\mathcal A$, $\co{C}(\Cl^-(\mathcal A))$ is the
  polarity $(\Cl^-(\mathcal A),\Cl^-(\mathcal A),\subseteq)$.
  Define
  $\epsilon_{\mathcal A}\subseteq A^-\times \Cl^-(\mathcal A)$ as
  the membership relation restricted to $A^-$ and $\Cl^-(\mathcal A)$.

  So for a family $\Xi\subseteq \Cl^-(\mathcal A)$, 
  $\epsilon_{\mathcal A}^\da(\Xi) = \bigcap_{X\in\Xi} X$. 
  For $X\subseteq A^-$,
  $\epsilon_{\mathcal A}^\ua(X) = \{Y\in \cl^-(\mathcal{X}) \st X\subseteq Y\}$.
  Hence $\epsilon_{\mathcal A}^\da \epsilon_{\mathcal A}^\ua(Y) = \cl_{\mathcal A}(Y)$ and \[\epsilon_{\mathcal A}^\ua\epsilon_{\mathcal A}^\da(\{X_i\}_i) = 
  \{Y\in \Cl^-(\mathcal{A})\st Y \supseteq \bigcap_i X_i\}.\] 
  The latter is clearly the Galois closure $\mathord{\subseteq^\ua}\mathord{\subseteq^\da}(\{X_i\}_i)$ 
  in the polarity $\co{C}(\Cl^-(\mathcal A))$.  So
  $\epsilon_{\mathcal A}$ is an isomorphism. 
  
  For naturality, suppose $\pfromto R{\mathcal A}{\mathcal B}$
  is a compatible relation. 
  Then $R\fatsemi \epsilon_{\mathcal B}$ and $\epsilon_{\mathcal A} \fatsemi \co{C}(\Cl^-(R))$
  can be compared as follows. For $a\in A^-$ and $Y\in \Cl^-(\mathcal B)$, it is the case that
  $a\mathrel{R\fatsemi \epsilon_{\mathcal B}} Y$
  if and only if $a\in R^\da {\mathcal B}^\ua(Y)$, which is equivalent to ${\mathcal B}^\ua (Y) \subseteq R[a]$. And $\co{C}(\Cl^-(R))^\da$ sends $Y$ to the principle downset in $\Cl^-(\mathcal A)$
  generated by $R^\da {\mathcal B}^\ua(Y)$.
  The incidence relation on $\co{C}(\Cl^-(\mathcal A))$ is $\subseteq$ restricted
  to Galois closed subsets of $A^-$, so $\co{C}(\Cl^-(\mathcal A))^\ua \co{C}(\Cl^-(R))^\da(Y)$ is just the upset in $\Cl^-(\mathcal A)$ generated by $R^\da {\mathcal B}^\ua(Y)$.
  Finally, $\epsilon_{\mathcal A}$ is membership, so
  $a\mathrel{\epsilon_{\mathcal A} \fatsemi \co{C}(\Cl^-(R))} Y$
  if and only if $a\in R^\da{\mathcal B}^\ua(Y)$.
  
  For a complete lattice $L$, $\Cl^-(\co{C}(L))$ consists precisely of
  the principle downsets of $L$. Since $L$ is a complete lattice,  $x\mapsto \da x$ is an isomorphism. 
  Naturality follows
from the simple calculation that for an $\INF$ morphism
$\fromto hML$, $\Cl^-(\co{C}(h))$ sends a principle downset $\da x$
to the principle downset $\da h(x)$. 
\end{proof}

\section{Complete Lattices}\label{sec:complete-lattices}

The duality between categories $\Pol$ and $\INF$
is actually a $2$-category fact because it preserves order on morphisms,
and of course $\fatsemi$-composition is monotonic on both sides.
This provides an easy way to locate the
morphisms in $\Pol$ that correspond to complete lattice homomorphisms.
Namely, they will be precisely the ``maps'', i.e., those morphisms
that possess a left adjoint. In addition, we find other, more direct characterizations.

\begin{theorem}\label{thm:cxt-adjoints}
  For a compatible relation $\pfromto R{\mathcal A}{\mathcal B}$ between polarities $\mathcal A$ and $\mathcal B$, define $R_*\subseteq B^-\times A^+$
as $R_* = \reln{{\mathcal B}^\da R^\ua {\mathcal A}^\da}$. Then the following are equivalent.
  \begin{enumerate}
  \item $\Cl^-(R)$ preserves arbitrary joins;
  \item there exists compatible $\pfromto S{\mathcal B}{\mathcal A}$ so that ${\mathcal A} \subseteq R\fatsemi S$ 
and $S\fatsemi R \subseteq {\mathcal B}$;
  \item there exists compatible $\pfromto S{\mathcal B}{\mathcal A}$ so that $R^\ua {\mathcal A}^\da = {\mathcal B}^\ua S^\da$;
  \item $R_*$ is compatible and $R^\ua {\mathcal A}^\da = {\mathcal B}^\ua R_*^\da$.
  \item $R_*$ is compatible and ${\mathcal B}^\da R^\ua = R_*^\da {\mathcal A}^\ua$.
  \item $R_*^\da = {\mathcal B}^\da R^\ua {\mathcal A}^\da$.
  \end{enumerate}
\end{theorem}
\begin{proof}
  Equivalence of (1) and (2) follows from the fact that the map
$\Cl^-(R)$ preserves
joins if and only if it has a left adjoint,
and that this adjoint must also preserve arbitrary meets. Hence it must be $\Cl^-(S)$ for some
compatible $S$. Conversely, if an $S$ satisfies the condition in (2), then $\Cl^-(S)$ 
is the desired adjoint to $\Cl^-(R)$.

[(6) implies (4)] This is immediate because $R$ is assumed to be compatible. 

[(4) implies (3)] Trivial.

[(3) implies (2)] From (3), ${\mathcal A}^\da \leq R^\da R^\ua {\mathcal A}^\da = R^\da {\mathcal B}^\ua S^\da$. So ${\mathcal A} \subseteq R\fatsemi S$.
And likewise, 
\begin{align*}
  (S\fatsemi R)^\da &= S^\da {\mathcal A}^\ua R^\da\\
  & = {\mathcal B}^\da {\mathcal B}^\ua S^\da {\mathcal A}^\ua R^\da\\
  & = {\mathcal B}^\da R^\ua {\mathcal A}^\da {\mathcal A}^\ua R^\da\\
  & = {\mathcal B}^\da R^\ua R^\da \leq {\mathcal B}^\da.
\end{align*}
                
[(2) implies (6)] By virtue of the Galois connection, ${\mathcal A}\subseteq R\fatsemi S$ if and only if ${\mathcal B}^\ua S^\da \leq R^\ua {\mathcal A}^\da$,
  which implies ${\mathcal B}^\da R^\ua {\mathcal A}^\da\leq {\mathcal B}^\da {\mathcal B}^\ua S^\da = S^\da$.
  Hence $R_*\subseteq S$. Also $S^\da = {\mathcal B}^\da {\mathcal B}^\ua S^\da = {\mathcal B}^\da R^\ua {\mathcal A}^\da$. 

  For the other inclusions,
  \begin{align*}
    S^\da &= S^\da {\mathcal A}^\ua {\mathcal A}^\da\\
    &\leq S^\da {\mathcal A}^\ua R^\da R^\ua {\mathcal A}^\da\\
    &\leq {\mathcal B}^\da R^\ua {\mathcal A}^\da\\
    &\leq \reln{{\mathcal B}^\da R^\ua {\mathcal A}^\da}^\da
  \end{align*}
  so $S\subseteq R_*$ and $S^\da \leq {\mathcal B}^\da R^\ua {\mathcal A}^\da$.

  Suppose (4). Then ${\mathcal B}^\da R^\ua = {\mathcal B}^\da R^\ua {\mathcal A}^\da {\mathcal A}^\ua = {\mathcal B}^\da {\mathcal B}^\ua R_*^\da {\mathcal A}^\ua = R_*^\da {\mathcal A}\ua$.
  Suppose (5). Then $R^\ua {\mathcal A}^\da = {\mathcal B}^\ua {\mathcal B}^\da R^\ua {\mathcal A}^\da = {\mathcal B}^\ua R_*^\da {\mathcal A}^\ua {\mathcal A}^\da = {\mathcal B}^\ua R_*^\da.$
\end{proof}

Thus we have a duality for complete lattices.

\begin{theorem}\label{thm:mapcxt=clat}
  The subcategory of $\Pol$ consisting of polarities and compatible relations
  $\pfromto R{\mathcal A}{\mathcal B}$ for which $\reln{{\mathcal B}^\da R^\ua {\mathcal A}^\da}^\da = {\mathcal B}^\da R^\ua {\mathcal A}^\da$
  is dually equivalent to the category $\CLat$.
\end{theorem}

\section{RS-frames}\label{sec:monos-epis}

There is a simple procedure for constructing certain
monomorphisms and epimorphisms in $\Pol$. 
Suppose $\mathcal A$ is a polarity and
$X\subseteq A^-$. 
Define the \emph{$-$-restriction} ${\mathcal A}\upharpoonright{X}$
as the polarity  $(X, A^+; \mathcal A\cap (X\times A^+))$. 
Similarly, for $\Xi\subseteq A^+$,
define the \emph{$+$-restriction} ${\mathcal A}\downharpoonright{\Xi}$ analogously, or if the reader 
prefers, as ${\mathcal A}\downharpoonright{\Xi} = ({\mathcal A}^\partial\upharpoonright{\Xi})^\partial$.  
One can check that
\begin{align*}
  ({\mathcal A}\upharpoonright{X})^\ua(X') &= \mathcal A^\ua(X') & \mbox{for $X'\subseteq X$}\\
  ({\mathcal A}\upharpoonright{X})^\da(\Xi) &= \mathcal A^\da(\Xi)\cap X & \mbox{for $\Xi\subseteq A^+$}
\end{align*}
So the relation ${\mathcal A}\upharpoonright{X}$ is a monomorphic compatible relation from
the polarity ${\mathcal A}\upharpoonright{X}$ to ${\mathcal A}$. Similarly, ${\mathcal A}\downharpoonright \Xi$
is an epimorphic compatible relation from ${\mathcal A}$ to ${\mathcal A}\downharpoonright \Xi$.

By the foregoing paragraph, every subset of $A^-$ determines a sub-object, but the converse
fails: there are monomorphisms (indeed, isomorphisms) in $\Pol$
for which the domain is not isomorphic to any
${\mathcal A}\upharpoonright{X}$.  Polarities
$P(A) = (A,A,\neq)$ provide examples of this.

Recall that in Gerhke's RS-frames, the S
stands for \emph{separating}, precisely in the sense of Lemma \ref{lem:separating}.
The $R$ stands for \emph{reduced}. Suppose we have a separating polarity $\mathcal A$. 
We can understand reduced polarities (in Gehrke's terms, reduced frames) as follows.

Restriction of $A^-$ to a subset $X$ produces a monomorphism into $\mathcal A$ that may actually be an isomorphism.
In that case, the inverse is given by the relation $\rel(\mathcal A)$
as a morphism from $\mathcal A$ to $\mathcal A\upharpoonright X$.

For a polarity $\mathcal A$ and $X\subseteq A^-$, say that ${\mathcal A}\upharpoonright{X}$ is a $-$ reduction
of $\mathcal A$ if ${\mathcal A}\upharpoonright{X}$ is  isomorphic. By dualizing,
we also have a notion of $+$ reduction.
Say that $\mathcal A$ is \emph{reduced} if it has no proper $-$ or $+$ reductions. 
In effect,
we can not remove anything from $A^-$ or $A^+$ while staying in the isomorphism class of $\mathcal A$.

Gehrke's RS-frames \cite{gehrke_generalized_2006} are precisely the separating, reduced polarities. 
The reader
may find it useful to chase the definition of ``reduced'' as it is found in the formal context literature
(essentially the same as Gehrke's) to see that it is precisely the one given here. 
The point is 
that our account of reduction gives a natural motivation for the definition.
 
Importantly,
there are polarities that are not isomorphic to any RS-frame. Consider the polarity $Q=(\QQ,\QQ,\leq)$.
For any two order dense subsets of $\QQ$, say $A$ and $B$, the usual back and forth proof that all countable
dense linear orders without endpoints are isomorphic adapts to show that $(A,B,\leq)$ is a proper
reduction of $Q$.

Without separation, there is no upper bound on the cardinality of polarities that are isomorphic to a given $\mathcal A$.
That is, there is no upper bound on the cardinality of polarities $\mathcal B$ for which $\mathcal A$ is
a reduction. 
But if we focus attention on separating polarities, 
    then it is quite easy to see that
$\co{C}(\Cl^-(\mathcal A))$ is the largest separating polarity isomorphic to $\mathcal A$.

\section{Completeness of $\Pol$}\label{sec:completeness}

Next, we turn our attention to the existence of limits and co-limits.  
The spirit of the last section is to 
demonstrate that a desired polarity may be constructed directly from the data we are given.
One would hope that the construction of limits and co-limits would follow a similar pattern. Of course,
we know that $\INF$ is (co)complete. So a co-limit in $\Pol$ can be found simply by sneaking over to $\INF$
(contravariantly), finding the corresponding limit there, and then sneaking back. 
The result can be quite unwieldy, as the functorial image $\Cl^-(\mathcal A)$ produces a complete lattice of subsets of $A^-$. We offer a purely \emph{polarity-theoretic}
construction that does not use this indirect method. 
The self-duality of $\Pol$ means that the 
construction of limits is the same as the construction of co-limits, so we concentrate
on limits here. Indeed, products and co-products coincide in $\Pol$ for very obvious reasons. They coincide in $\INF$, of course. 
But that fact is not so obvious directly in $\INF$.

To describe polarity constructions, we will need standard notation for products and coproducts (disjoint unions) in sets.
As usual $\prod_{i}X_i$ denotes a cartesian product where the set of indices is omitted. 
For $\mathbf x\in \prod_{i}X_i$, and $k$ an index, $\mathbf x_k$ is the projection.
And $\sum_{i}X_i$ is the disjoint union.
For index $k$, and $x\in X_k$, ${}_k x \in \sum_{i}X_i$ is the insertion of $x$ in the disjoint union.

For a family $\{{\mathcal A}_i\}_{i}$ of polarities, 
    define the product polarity $\prod_i\mathcal A_i$ by the following data.
    \begin{align*}
       (\prod_i\mathcal A_i)^+ &= \sum_i A_{i+}\\
       (\prod_i\mathcal A_i)^- &= \sum_i A_{i-}\\
       {}_j a (\prod_i\mathrel{A_i}) {}_k \alpha&\iff j = k \to a\mathrel{\mathcal A_k} \alpha 
    \end{align*}

\begin{lemma}\label{lem:cxt-products}
For any family of polarities $\{\mathcal A_i\}_i$, 
    the polarity $\prod_i \mathcal A_i$ is the product in $\Pol$. 
\end{lemma}
\begin{proof}
For index $k$, define $P_k$ to be the corestricted from $\prod_i \mathcal A_i$ 
    to $\mathcal A_k$.
That is, ${}_ja\mathrel{P_k}\alpha$ if and only if $j=k$ implies $a\mathrel{\mathcal A_k}\alpha$.

Evidently, $P_k$ is compatible with $\prod_i \mathcal A_i$ on the right and with $\mathcal A_k$ on the left. 
For a polarity $\mathcal B$ and a family of compatible relations $\{\pfromto {R_i}{\mathcal B}{\mathcal A_i}\}_i$, define the relation $T$ by $b\mathrel{T}{}_i\alpha$ if and only if $b\mathrel{R_i}\alpha$.
It is routine to show that this is the unique ``tupling'' morphism from $\mathcal B$ to $\prod_i \mathcal A_i$ for which $T\fatsemi P_k = R_k$
for each index $k$.
\end{proof}

The coproduct is exactly the same construction. Insertions into the coproduct are duals of the projections.

\begin{lemma}
	For compatible relations $\pfromto {R,S}{\mathcal A}{\mathcal B}$, an equalizer is 
	$\mathcal A\upharpoonright E$ where $E = \{a\in A^- \st R[a] = S[a]\}$.
\end{lemma}

We leave the verification to the reader.

The two preceding lemmas plus Lemma~\ref{lem:cxt-selfdual}, sum up to the following.

\begin{theorem}\label{thm:cxt-complete}
	The category $\Pol$ is complete and cocomplete.
\end{theorem}

\section{Tensors and internal homs}\label{sec:tensors}

Though $\Pol$ is not cartesian closed (cannot be because it is self dual and is not trivial), it is $*$-autonomous, and therefore has an internal hom functor.
Recall that one characterization of a $*$-autonomous category $\cat{A}$ is that it is symmetric monoidal possessing a contravariant functor ${}^*$ satisfying 
\[
\cat{A}(A\otimes B,C^*)\simeq \cat{A}(A,(B\otimes C)^*)
\]
naturally.
In that situation, the internal hom is $A\multimap B = (A\otimes B^*)^*$.

We have our candidate contravariant functor $\mathcal A\mapsto \mathcal A^\partial$.
A suitable tensor product in $\Pol$, once again defined directly from the data in the constituent polarities, can be defined by generalizing the notion of a standard polarity $\tilde{\mathcal A}$ that we encountered in Lemma~\ref{lem:standardization}. 

For polarities $\mathcal A$ and $\mathcal B$, say that $T\subseteq A^-\times B^-$ is \emph{stable} if it is compatible with $\mathcal A$ 
on the left, and with $\mathcal B^\partial$ on the right.
That is, $X\times Y\subseteq R$ implies $\cl_{\mathcal A}(X)\times \cl_{\mathcal B}(Y)\subseteq R$.
Let $s(T)$ denote the smallest stable subset containing $T$.

\begin{lemma}\label{lem:basic-stable-sets}
For any polarities $\mathcal A$ and $\mathcal B$, 
and sets $X\subseteq A^-$, and $Y\subseteq B^-$, the set $\cl_{\mathcal A}(X)\times \cl_{\mathcal B}(Y)$ is itself stable, so it is the least stable set containing $X\times Y$.
\end{lemma}
\begin{proof}
    This follows directly from the fact that $\cl_{\mathcal A}$ is monotonic and idempotent.
\end{proof}

Since $s$ is a closure on $A^-\times B^-$,
    we can define $\mathcal A\otimes \mathcal B$ to be
    the standard polarity $\mathsf A_s$.
Specifically, $(a,b) \mathrel{(\mathcal A\otimes \mathcal B)} T$ if and only if $(a,b)\in s(T)$.
Keep in mind that $s(T)$ concretely is a compatible relation 
$\pfromto {s(T)}{\mathcal A}{\mathcal B^\partial}$.

To extend $\otimes$ to a functor, 
consider compatible relations $\pfromto{Q}{\mathcal A}{\mathcal B}$
    and $\pfromto S{\mathcal C}{\mathcal D}$.
    Then $S^\partial$ is compatible from $\mathcal D^\partial$ to $\mathcal C^\partial$. 
And by definition, for any $R\subseteq B^-\times D^-$, 
    $s(R)$ is compatible from $\mathcal B$ to $\mathcal D^\partial$.
    Thus, $Q\fatsemi s(R) \fatsemi S$ is a well-defined (compatible) relation from $\mathcal A$ to $\mathcal C^\partial$.
    
Define $Q\otimes S\subseteq (A^-\times C^-)\times \powerset(B^-\times D^-)$ to be the relation 
    $(a,c) \mathrel{(Q\otimes S)} R$ if and only 
    $a\mathrel{Q\fatsemi s(R)\fatsemi S^\partial} c$. 
As in the proof of Lemma~\ref{lem:cxt-monos}, let $\II = (\{\bullet\},\{\bullet\},\emptyset)$.

\begin{theorem}
    The constructions $\otimes$ and ${}^\partial$, and the object $\II$ make $\Pol$ a $*$-autonomous category.
\end{theorem}
\begin{proof}
    Because $Q\fatsemi s(R) \fatsemi S^\partial$ is a compatible 
relation, and $\cl_{\mathcal A\otimes \mathcal C^\partial}$ is compatibilization, $Q\otimes S$ is compatible with $\mathcal A\otimes \mathcal C$ on the left. 
For compatibility on the right,
observe that for a subset $\mathbf R\subseteq\powerset(B^-\times D^-)$, the closure $\cl^{\mathcal B\otimes \mathcal D}(\mathbf R)$ 
consists of all $R\subseteq B^-\times D^-$ for which 
$\bigcap_{R'\in \mathbf R} s(R') \subseteq s(R)$.
Hence, $(Q\otimes S)^\da(\mathcal R) \subseteq (Q\otimes S)^\da(\cl^{\mathcal B\otimes \mathcal D}(\mathbf R))$. 

To see that $\otimes$ is a functor, note that $((P\fatsemi Q)\otimes (T\fatsemi S))^\intercal[R]$ is the relation $P\fatsemi Q\fatsemi s(R)\fatsemi (T\fatsemi S)^\partial$, and $((P\otimes T)\fatsemi (Q\otimes S))^\intercal[R]$ is the relation 
$P\fatsemi (Q\fatsemi s(R)\fatsemi S^\partial)^*\fatsemi T^\partial$. But $Q\fatsemi s(R) S^\partial$ is already compatible,
so in this, the compatibilization does nothing.

Obviously, $\mathcal A\otimes \mathcal B$ is naturally isomorphic to $\mathcal B\otimes \mathcal A$ by the relation $(a,b)\Gamma R$ if and only if $(b,a) \mathrel{\mathcal B\otimes \mathcal A} R$.

Associativity follows from showing that $T\subseteq (A^-\times B^-)\times C^-$ is stable in $\mathcal A\otimes \mathcal B)\otimes \mathcal C))$ if and only if 
\begin{equation}\label{eqn:stability}
(X\times Y)\times Z\subseteq T \qquad \implies 
(\cl_{\mathcal A}(X)\times \cl_{\mathcal B}(Y))\times \cl_{\mathcal C}(Z) \subseteq T.\tag{*}
\end{equation}
Clearly, if $T$ is closed, then the implication (\ref{eqn:stability}) holds.
Suppose (\ref{eqn:stability}) holds, and $W\times Z\subseteq T$. 
Then all $X$ and $Y$, $X\times Y\subseteq W$ implies $(\cl_{\mathcal A}(X)\times \cl_{\mathcal B}(Y))\times \cl_{\mathcal C}(Z) \subseteq T$.
But $\cl_{\mathcal A\otimes \mathcal B}(W)$ is the union of all $\cl_{\mathcal A}(X)\times \cl_{\mathcal B}(Y)$ for $X\times Y\subseteq W$, and so $\cl_{\mathcal A\otimes \mathcal B}(W)\times \cl_{\mathcal C}(Z) \subseteq T$.

Now, $\II$ is a unit for $\otimes$ 
because the lower sets for $\II\otimes \mathcal A$ and $\mathcal A\otimes \II$ are
$\{\bullet\}\times A^-$ and $A^-\times \{\bullet\}$, and $\II\otimes \mathcal A$ and $\mathcal A\otimes \II$ are both essentially just the standardization of $\mathcal A$.

A compatible relation $\pfromto R{\mathcal A\otimes \mathcal B}{\mathcal C^\partial}$, concretely, is an element of $\Cl^-((\mathcal A\otimes B)\otimes \mathcal C)$. 
The bijection 
$(A^-\times B^-)\times C^- \simeq A^-\times (B^-\times C^-)$ sends $R$ to an element of $\Cl^-(\mathcal A\otimes (\mathcal B\otimes \mathcal C))$, which is a compatible relation from $\mathcal A$ to $(\mathcal B\otimes C)^\partial$.
\end{proof}

\section{Future work}
\label{sec:some_correspondence_theory}

Ern\'e, in \cite{erne_distributive_1993}, 
    investigates characterizations of polarities $\mathcal A$ 
    for which $\Cl^-(\mathcal A)$ satisfies various 
    distributivity laws.
In nice cases, the result is a condition expressible 
    in first-order using the two-sorted vocabulary 
    of a single binary relation. 

From the point of view that polarities are generalized 
    Kripke frames (not to be confused with the frames 
    of point-free topology), Ern\'e's work points toward 
    a sort of correspondence theory for polarities, 
    in which a property of the lattices $\Cl^-(\mathcal A)$ 
    that involves quantifying over closed sets, 
    and in cases such as infinite distributivity, 
    over sets of sets of closed sets, 
    is actually a elementary property of polarities.

The development of such a correspondence theory 
    would begin with two things. 
First, one would like a first-order correspondent 
    for frame homomorphisms, i.e., a sentence in the first-order 
    language of four sorts $A^-$, $A^+$, $B^-$, $B^+$, 
    and three binary relations $\mathbf A$, $\mathbf R$, 
    and $\mathbf B$ that encodes the ``frame' condition'' 
    that the interpretation of  $\mathbf R$ is compatible 
    with the interpretations of $\mathbf A$ and $\mathbf B$.
Second, one would like to have generalize this, for example, 
    to quantales in such a way that various 
    (unary) modal operators also fit into the generalization.

In a subsequent paper, currently in progress, we develop 
    this general elementary correspondence theory for polarities. 
For example, an elementarily definable category 
    of frame polarities and morphisms is equivalent via 
    $\Cl^-(-)$ to the category $\cat{Loc}$ of locales 
    (point-free topologies) 
    and localic (point-free continuous) maps. 
This opens the possibility for investigating traditional 
    locale-theoretic topics such as separation, 
    the structure of the frame of sublocales, uniformity, 
    and so on, in first-order terms.

In other work being pursued in parallel to this paper, 
    we consider topological extensions of polarities 
    that capture (non-distributive) lattices, 
    as opposed to general complete lattices.

Another area that warrants more investigation is a connection 
    to Giovanni Sambin's development of formal topology via 
    what he calls basic pairs \cite{sambin_points_2003}. 
Putting aside foundational issues (for Sambin the word ``set'' 
    does not encompass what the reader likely means), 
    a Sambinian \emph{basic pair} is precisely a polarity.
A morphism between basic pairs is a pair of functions, 
    satisfying slightly weaker conditions that in Ern\'e's 
    or Goldblatt's work, but for which equality is defined 
    more liberally.
In effect (classically), basic pair morphisms are 
    equivalence classes of pairs of functions. 
We have reasons to believe that these equivalence classes 
    correspond to compatible relations. 
If that bears out, we will have a very satisfying bridge 
    between Ern\'e, Goldblatt, and our approach.

\bibliographystyle{abbrv} 
\bibliography{references}

\end{document}